\documentclass{amsart}

\usepackage{rorabaugh}

\title{Graph cover-saturation}
\author{Danny Rorabaugh, Queen's University}

\newcommand{\cs}{\operatorname{csat}}
\renewcommand{\ss}{\operatorname{ssat}}
\newcommand{\ws}{\operatorname{wsat}}
\newcommand{\sat}{\operatorname{sat}}
\newcommand{\ex}{\operatorname{ex}}
\newcommand{\Cs}{\operatorname{Csat}}
\newcommand{\Sat}{\operatorname{Sat}}
\newcommand{\Ex}{\operatorname{Ex}}
\renewcommand{\o}[1]{\overline{#1}}
\renewcommand{\u}[1]{\underline{#1}}

\begin{document}

\begin{abstract}
	Graph $G$ is \emph{$F$-saturated} if $G$ contains no copy of graph $F$ but any edge added to $G$ produces at least one copy of $F$. 
	One common variant of saturation is to remove the former restriction: $G$ is \emph{$F$-semi-saturated} if any edge added to $G$ produces at least one new copy of $F$. 
	In this paper we take this idea one step further. 
	Rather than just allowing edges of $G$ to be in a copy of $F$, we require it: $G$ is \emph{$F$-covered} if every edge of $G$ is in a copy of $F$. 
	It turns out that there is smooth interaction between coverage and semi-saturation, which opens for investigation a natural analogue to saturation numbers. 
	Therefore we present preliminary cover-saturation theory and structural bounds for the cover-saturation numbers of graphs. 
	We also establish asymptotic cover-saturation densities for cliques and paths, and upper and lower bounds (with small gaps) for cycles and stars. 
\end{abstract}

 \date{\today}
 

\maketitle

\section{Introduction to Saturation}

We begin with a brief introduction to graph saturation before defining the concept of coverage and its corresponding saturation variant. 
Section~2 establishes preliminary theory for cover-saturation, then we prove several structural bounds on cover-saturation numbers in Section~3. 
Section~4, the final main part of this paper, investigates cover-saturation numbers for specific classes of graphs, including paths, cycles, and stars. 
We end with a discussion of several open directions for further study. 
All graphs are assumed to be simple, finite, and undirected. 

%\sat
\begin{defn}
	Graph $G$ is \emph{$F$-free} provided there is no subgraph of $G$ isomorphic to graph $F$. 
	We say $G$ is \emph{$F$-saturated} if $G$ is $F$-free and $G+e$ is not $F$-free for any edge $e$ in the complement of $G$. 
	
	The \emph{saturation number of $F$}, $\sat(n,F)$, is the fewest number of edges in an $F$-saturated graph on $n$ vertices. 
\end{defn}

The first result in graph saturation was given by Alexander Zykov~\cite{Zy-49} in Russian in 1949 and independently by Erd\H{o}s, Hajnal, and Moon~\cite{EHM-64} in English in 1964. 
They found the saturation number of a clique:
\begin{equation*}
	\sat(n,K_r) = (r-2)n - \binom{r-1}{2}.
\end{equation*}

%Since that time, the saturation number and numerous variants have been studied on many classes of graph and hypergraphs. 

%%%%%%%%%%%
\subsection{Pseudo-Saturation}
%\ws
Perhaps the first variation of saturation to be studied was weak saturation, introduced in 1967 by B\'ela Bollob\'as~\cite{Bo-68}. 
It is closely related to bootstrap percolation. 
\begin{defn}
	Graph $G$ is \emph{weakly $F$-saturated} provided the edges of the complement of $G$ can be ordered   $e_1, e_2, \ldots, e_{\ell}$ so that when we add the edges to $G$ one at a time, $G_0 = G$ and $G_i = G_{i-1} + e_i$ for $1\leq i \leq \ell$, then the number of copies of graph $F$ in $G_{i+1}$ is strictly greater than the number of copies of $F$ in $G_i$ for all $i <\ell$. 
	
	The \emph{weak saturation number of $F$}, $\ws(n,F)$, is the fewest number of edges in a weakly $F$-saturated graph on $n$ vertices. 
\end{defn}	

The most relevant variant of saturation to the present work is what was originally called strong saturation. 
However, some authors~(e.g.,~\cite{KS-16} and~\cite{MS-15}) have used the phrase ``strong saturation'' in reference to the usual saturation simply to contrast it with weak saturation. 
Thus, to avoid ambiguity, we follow the example of F\"uredi and Kim~\cite{FK-12} and use ``semi-saturation'' in place of ``strong saturation.''
\begin{defn}
	Graph $G$ is \emph{$F$-semi-saturated} if for any edge $e$ in the complement of $G$ the graph $G+e$ contains more copies of graph $F$ than are in $G$. 
	
	The \emph{semi-saturation number of $F$}, $\ss(n,F)$, is the fewest number of edges in an $F$-semi-saturated graph on $n$ vertices. 
\end{defn}

With this definition, we can restate the definition of $F$-saturation simply as: $F$-free and $F$-semi-saturated. 

\begin{exmp} \label{sat}
\begin{enumerate}[(a)]
	\item The Tur\'an graph, or balanced complete $r$-partite graph, $T(n,r)$ is $K_{r+1}$-saturated for every $n \geq r$. 
	\item Every graph $F$ is weakly $K_2$-saturated.
	\item The clique $K_n$ is vacuously $F$-semi-saturated for every $F$.
\end{enumerate}
\end{exmp}

See the 2011 survey by Faudree, Faudree, and Schmitt~\cite{FFS-11} for a more comprehensive coverage of known results and open problems about saturation and several variations thereof. 

%%%%%%%%%%%
\subsection{Anti-Saturation}

It is also worth mentioning Tur\'an theory, to which saturation is considered a dual or opposite theory. 
Whereas $\sat(n,F)$ is the minimum number of edges in an $n$-vertex $F$-saturated graph, 
 the \emph{extremal number} $\ex(n,F)$ is the maximum number of edges in an $n$-vertex $F$-saturated graph. 
In 1941, P\'al Tur\'an~\cite{Tu-41} proved (with different notation) that $\lim_{n\rightarrow \infty} \ex(n,K_r)/\binom{n}{2} = \frac{r-1}{r}$. 
This was generalized for all graphs in 1946 by Erd\H{o}s and Stone~\cite{ES-46}, 
 who proved that $\lim_{n\rightarrow \infty} \ex(n,F) / \binom{n}{2} = \frac{\chi(F)-2}{\chi(F)-1}$, where $\chi(F)$ is the chromatic number of $F$. 
 
So the extremal number of a (non-bipartite) graph is on the order of $n^2$, but $\sat$, $\ws$, and $\ss$ are on the order of $n$ (or bounded) for every graphs. 
As stated by Zsolt Tuza~\cite{Tu-92} in 1992: ``in contrast with the Tur\'an numbers (in which the chromatic number as a ``global parameter'' is essential), the growths of [saturation numbers] depend on some \emph{local} parameters[...]."

%%%%%%%%%%%
\subsection{Asymptotics}

Since saturation numbers of $F$ grow no faster than some constant multiple of $n$, 
 it is natural to divide by $n$ and take the limit, but that limit is not known to always exist.
Zsolt Tuza~\cite{Tu-88} conjectured in 1988 that $\lim_{n\to\infty} \frac{\sat(n,F)}{n}$ exists for all $n$. 
For convenience, we will use $\sat(F)$ when this limit exists and $\u{\sat}(F)$ and $\o{\sat}(F)$ for the limit infimum and limit supremum, respectively. 
In 1991, Truszczy\'nski and Tuza~\cite{TT-91} made the following progress toward the latter's conjecture: 
 If $\u{\sat}(F) < 1$, then $\sat(F) = 1 - \frac{1}{p}$ for some positive integer $p$. 
They also gave a characterization of all such graphs.

%%%%%%%%%%%%%%%%%%%%%%%%%
\section{Coverage and Saturation}

The idea of graph semi-saturation was to lift one of the restrictions imposed by saturation: 
 The edges of an $F$-semi-saturated graph are allowed to be in a copy of $F$. 
Here we consider a concept we call coverage, where edges are not only allowed but required to be in a copy of $F$. 
This leads to a theory of cover-saturation and an analogous saturation number. 

\begin{defn}
Graph $G$ is \emph{$F$-covered} provided every edge of $G$ is in a subgraph of $G$ isomorphic to graph $F$. 
\end{defn}

\begin{exmp} \label{cov}
	The clique $K_n$ is $F$-covered for any graph $F$ with at least one edge and at most $n$ vertices.
\end{exmp}

\begin{defn}
Graph $G$ is \emph{$F$-cov-sat} provided $G$ is both $F$-covered and $F$-semi-saturated. 
\end{defn}

\begin{exmp} \label{covsat}
Every graph with at least one edge is $K_2$-cov-sat. 
\end{exmp}

Henceforth, to avoid trivial counterexamples, we assume graphs have at least one edge. 
Before we introduce the cov-sat analogue of saturation numbers of graphs, let us make a few observations about coverage and saturation. 

\begin{fact} \label{trans}
\begin{enumerate}[(a)]
	\item Coverage is a transitive graph relation:\\ If $G$ is $F$-covered and $H$ is $G$-covered, then $H$ is $F$-covered.
	\item If $G$ is $F$-covered and $H$ is $G$-semi-saturated then $H$ is $F$-semi-saturated.
	\item From (a, b): If $G$ is $F$-covered and $H$ is $G$-cov-sat, then $H$ is $F$-cov-sat.
	\item From (c), cover-saturation is a transitive graph relation:\\ If $G$ is $F$-cov-sat and $H$ is $G$-cov-sat, then $H$ is $F$-cov-sat.
\end{enumerate}
\end{fact}

\begin{fact} \label{deltaFlem} Let $\delta(G)$ denote the minimum degree of graph $G$. 
	\begin{enumerate}[(a)]
		\item If $G$ is $F$-covered and $\delta(G) \geq 1$, then $\delta(G) \geq \delta(F)$. 
		\item If $G$ is $F$-semi-saturated, then $\delta(G) \geq \delta(F) - 1$. 
	\end{enumerate}
\end{fact}

Semi-saturation and coverage also relate naturally to connectivity. 

\begin{lem} %assumes $G$ has multiple vertices
	If every connected component of $F$ is $k$-connected and $G$ is $F$-semi-saturated, then $G$ is $(k-1)$-connected.
\end{lem}

\begin{proof}
	For any non-adjacent vertices $x$ and $y$ in $G$, $G+xy$ has a copy of $F$ that uses $xy$. 
	Since each component of $F$ is $k$-connected, by Menger's theorem, there are $k$ vertex-disjoint $x$-$y$-paths in $F$, and thus too in $G+xy$. 
	Hence we have $k-1$ vertex disjoint $x$-$y$-paths in $G$. 
	Since this is true for arbitrary non-adjacent $x$ and $y$ in $G$, $G$ is $(k-1)$-connected. \end{proof}

\begin{lem} \label{kedgeconnected} %assumes $G$ has multiple vertices
	If every connected component of $F$ is $k$-edge-connected and $G$ is $F$-cov-sat, then $G$ is $(k-1)$-edge-connected.
\end{lem}

\begin{proof}
	Adjacent vertices in $G$ are connected by $k$ edge-disjoint paths in $G$ because $G$ is $F$-covered. 
	Non-adjacent vertices in $G$ are dealt with using the fact that $G$ is $F$-semi-saturated, in the fashion of the previous proof. 
\end{proof}

%There are $k$-regular $k$-connected graphs, so this doesn't immediately give us anything better than the min-degree bound at the start of the next section.\\

The natural interaction between coverage and semi-saturation demonstrated in Fact~\ref{trans}(b) and Lemma~\ref{kedgeconnected} is part of our motivation for introducing coverage into the rich area of graph saturation. 
We quantify how small of graphs can be $F$-cov-sat, the analogue to $\sat(n,F)$, with the following cov-sat values.

\begin{defn}
For $|F| \leq n$, the \emph{cov-sat number} of $F$ is 
\[ \cs(n,F) = \min(|E| : |V| = n, G = (V,E) \text{ is } F \text{-cov-sat}) ; \]
and the \emph{(asymptotic) cov-sat density} of $F$ is
\[ \cs(F) = \lim_{n\rightarrow \infty} \frac{\cs(n,F)}{n}  \]
when the limit exists. 
We use $\u{\cs}(F)$ and $\o{\cs}(F)$ for the limit infimum and limit supremum, respectively. 
\end{defn}

We will refer to graphs that realize that cov-sat number of $F$ as \emph{extremal} for $F$.
Graphs in a family that realizes the cov-sat density of $F$ are \emph{asymptotically extremal} for $F$, though the individual graphs may not be extremal for $F$. 

A few basic bounds on the cov-sat values follow from the above definitions and observations.

\begin{thm} \label{deltaF}
	If every connected component of $F$ has at least two edges, then $\u{\cs}(F) \geq \delta(F)/2$.
\end{thm}

\begin{proof}
	This follows immediately from Fact~\ref{deltaFlem} in the case that $\delta(F) \geq 2$. 
	If $\delta(F) = 1$, observe that any $F$-cov-sat graph has at most one isolated vertex, since $F$ has no isolated edges. 
\end{proof}

\begin{fact} \label{FcovG}
	If $G$ is $F$-covered then $\cs(n,F) \leq \cs(n,G)$ for all $n \geq |G|$.
\end{fact}

\begin{proof}
	Fact~\ref{trans}(c)
\end{proof}

%\begin{fact} \label{FcovK}
%	If $|F| \leq s$, then $\cs(n,F) \leq \cs(n,K_s)$.
%\end{fact}
%
%\begin{proof}
%	Fact~\ref{FcovG} and Examples~~\ref{sat}(c) and~\ref{cov}.
%\end{proof}

%%%%%%%%%%%%%%%%%%%%%%%%%
\section{Upper Bounds on $\cs(F)$} 

%\subsection{Upper bounds}

%\begin{lem} %This will be useful for proving the 
%	If $F$ does not have a bridge, then $\cs(F)$ is realized by a family of graphs $G_n$ of the following form: A clique $C_n = K_\ell$ of some bounded order $\ell = \ell_n$ with a fixed subset $S$ of $s_F$ vertices, and the rest of the vertices partitioned into copies of the same graph $H_F$, every copy connected to $S$ in the same way. (Note that variation of $\ell_n$ depends only on $n \mod{|H_F|}$.)
%\end{lem}
%
%In the proof of Theorem~\ref{csKr}: When $F = K_r$, the extremal graphs are with $\ell = r-2$ if $n \equiv r \mod{2}$ and $\ell = r-1$ if $n \equiv r \mod{2}$; $s_F = r-2$; $H_F = K_2$; and every vertex of every copy of $H_F$ is connected to every vertex in $S$. 

From Fact~\ref{FcovG} and Examples~~\ref{sat}(c) and~\ref{cov}, we have that $\o{\cs}(F) \leq \o{\cs}(K_r)$ for all graphs $F$ with at most $r$ vertices. 
Whereas the Zykov-Erd\H{o}s-Hajnal-Moon theorem gives us that $\sat(K_r) = r-2$, we find a slightly larger value for $\cs(K_r)$, which gives us a general upper bound for cov-sat density. 

\begin{thm} \label{csKr}
For the clique $K_r$ on $r \geq 3$ vertices,
\[ \cs(K_r) = r - \frac{3}{2} . \]
\end{thm}

\begin{proof} 
For two graphs $G$ and $H$, let $G+H$ be the graph formed by taking a disjoint union $G \cup H$ and adding an edge $gh$ for every $(g,h) \in V(G)\times V(H)$. 
Set $M^*_{2k} = kK_2$ (the disjoint union of $k$ edges) and $M^*_{2k+1} = K_3 \cup (k-1)K_2$. 
Let $C(n,r) = K_{r-2} + M^*_{n-(r-2)}$, pictured in Figure~\ref{fig:csKr}. 
(This is an homage to the Tur\'an graph, with ``C'' for ``Coverage'' in place of ``T''.) 
First observe that $C(n,r)$ is $K_r$-covered and $K_r$-semi-saturated. 
Therefore, the number of edges in $C(n,r)$ gives and upper bound on the cov-sat number of $K_r$:
\[ \cs(K_r,n) \leq \binom{r-2}{2} + (r-2)(n-r+2) + \ceil{\frac{n - r + 2}{2}} + (n-r)\%2, \]
where $a\%b$ is the remainder when $a$ is divided by $b$.
Taking the limit of this bound divided by $n$, we have $\o{\cs}(K_r) \leq (r-2) + 1/2$. 

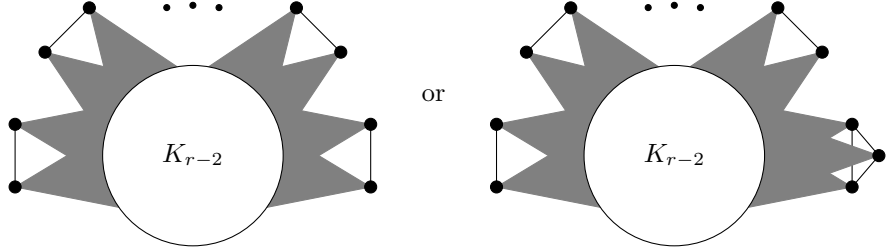
\begin{figure}[ht]
\begin{tikzpicture}[scale=.8]
	\foreach \x in {180, 135, 45, 0} {
		\filldraw[black!50!white] (\x-35:1.5) -- (\x+15:1.5) -- (\x-10:3) -- (\x-35:1.5);
		\filldraw[black!50!white] (\x-15:1.5) -- (\x+35:1.5) -- (\x+10:3) -- (\x-15:1.5);
		\draw (\x-10:3)--(\x+10:3);
		\filldraw (\x+10:3) circle(1mm);
		\filldraw (\x-10:3) circle(1mm);
	}
	\foreach \x in {35,45,55} {
		\filldraw (45+\x:2.5) circle(.5mm);
	}
	\filldraw[fill=white] (0:0) node{$K_{r-2}$} circle(1.5);
	
	\draw (4,1) node{or};
	
	\begin{scope}[xshift=8cm]
%		\filldraw[black!50!white] (-25:1.5) -- (25:1.5) -- (0:3.5) -- (-25:1.5);
		\foreach \x in {180, 135, 45, 0} {
			\filldraw[black!50!white] (\x-35:1.5) -- (\x+15:1.5) -- (\x-10:3) -- (\x-35:1.5);
			\filldraw[black!50!white] (\x-15:1.5) -- (\x+35:1.5) -- (\x+10:3) -- (\x-15:1.5);
			\draw (\x-10:3)--(\x+10:3);
			\filldraw (\x+10:3) circle(1mm);
			\filldraw (\x-10:3) circle(1mm);
		};
%		\draw (-10:3)--(0:3.5)--(10:3);
%		\filldraw (0:3.5) circle(1mm);
		\foreach \x in {35,45,55} {
			\filldraw (45+\x:2.5) circle(.5mm);
		}
		
		\filldraw[black!50!white] (0:3.4) -- (-30:1.5) -- (30:1.5) -- (0:3.4);
		\draw (10:3)--(0:3.4)--(-10:3);
		\filldraw (0:3.4) circle(1mm);
		
		\filldraw[fill=white] (0:0) node{$K_{r-2}$} circle(1.5);
	\end{scope}
\end{tikzpicture}
\caption{Asymptotically extremal graphs $C(n,r)$ for $K_r$, with $n \equiv r \mod{2}$ on the left and $n \equiv r+1 \mod{2}$ on the right.} \label{fig:csKr}
\end{figure}

To show this is best possible, we will get a lower-bound count on the number of edges in two different ways. 
Let $\delta = \delta(G)$ be the minimum degree of a $K_r$-cov-sat graph $G = (V,E)$, $n = |V|$, and $m = |E|$. 
Clearly $m \geq \delta n/2$. 

For any vertex $v \in V$, if $uv \not\in E$, then there are at least $r-2$ edges from $u$ to $N(v)$ since $G$ is $K_r$-semi-saturated. 
Moreover, since $G$ is $K_r$-covered, $u$ is incident to at least $r-1$ edges, but its $(r-1)$-th edge could be to another vertex outside $N(v)$. 
(See Figure~\ref{fig:csKr-lower}.)
Therefore, with $d = d(v)$,
\begin{eqnarray*}
	m & \geq & d + (n - 1 - d)\!\left[(r-2) + \frac{1}{2}\right] \\
	& = & (n-1)\!\left(r - \frac{3}{2}\right) - d\!\left(r-\frac{1}{2}\right) .
\end{eqnarray*}

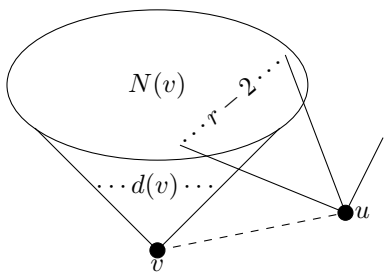
\begin{figure}[ht]
\begin{tikzpicture}
	\filldraw (0,0) circle(1mm) node[below]{$v$};
	\draw (135:2.3)--(0,0)--(45:2.3);
	\draw (0,2.2) ellipse(2 and 1) node{$N(v)$};
	\draw (0,.85) node{$\cdots d(v) \cdots$};
	\filldraw (2.5,.5) circle(1mm) node[right]{$u$};
	\draw (.3,1.4)--(2.5,.5)--(1.7,2.6);
	\draw(1,2)node[rotate=40]{$\cdots r-2 \cdots$};
	\draw (2.5,.5)--(3,1.5);
	\draw[dashed](0,0)--(2.5,.5);
\end{tikzpicture}
\caption{A look inside a $K_r$-cov-sat graph.} \label{fig:csKr-lower}
\end{figure}

Now we have that 
\[ m \geq \max\left( \frac{\delta n}{2}, (n-1)\!\left(r - \frac{3}{2}\right) - \delta\!\left(k-\frac{1}{2}\right) \right) .\]

For fixed $n$ and $r$, the former is increasing with $\delta$ and the latter is decreasing with $\delta$ so the maximum is minimized when 
\[ \frac{\delta n}{2} = (n-1)\!\left(r - \frac{3}{2}\right) - \delta\!\left(r-\frac{1}{2}\right),\]
which gives 
\[ \delta = \frac{(n-1)(2r-3)}{n + 2r - 5} . \]
This approaches $2r-3$ as $n$ approaches $\infty$, so 
\[ m \geq \frac{n(2r-3 + o(1))}{2} .\]

\end{proof}

\begin{cor} %assumes $F$ has an edge
	For graph $F$, $\o{\cs}(F) \leq |F| - \frac{3}{2}$. 
\end{cor}

%\begin{proof}
%Combine Fact~\ref{FcovK} and Theorem~\ref{csKr}. 
%\end{proof}

%%%%%%%%%%%%%%%%%%%%%%%%%%
%\section{Connectivity Bounds} 
%
%In this section, we establish several bounds on $\cs(F)$ depending on various local and global connectivity properties of $F$.
In a graph, a \emph{bridge} is an edge whose removal increases the number of connected components of the graph. 

\begin{thm} \label{bridge}
	If $F$ has a bridge, then $\o{\cs}(F) \leq \frac{|F|-1}{2}$.

	Moreover, if $F$ has a bridge $b$ such that every component of $F-b$ has at most $r$ vertices, then $\o{\cs}(F) \leq \frac{r(r-1) + 1}{2r}$. 
\end{thm}

\begin{proof}
	This upper bound is demonstrated by $G$ consisting of disjoint copies of the clique $K_{|F|}$ (and one clique $K_s$ with $|F| \leq s < 2|F|$). 
	$G$ is $F$-covered, because each component is $F$-covered. 
	$G$ is $F$-semi-saturated, because adding an edge between two of the cliques creates at least one copy of $F$ (many, in fact) since $F$ has a bridge. 

	With a bound of $r$ on the number of vertices in each component of $F-b$, we can improve the upper bound on $\cs(F)$. 
	This is demonstrated by disjoint copies of pairs of the clique $K_r$ where the two cliques in a pair are connected by a single edge.
\end{proof}

%%%%The following fact and corollary are weak instances of Theorem~\ref{deltaF}.
%\begin{fact}
%	If $F$ has no bridges and $G$ is $F$-covered, then $G$ has no bridges. 
%\end{fact}
%
%\begin{cor} %assumes $F$ has an edge
%	If $F$ has no bridges, then $\cs(F)\geq 1$.
%\end{cor}
%
%\begin{proof}
% 	Any $F$-covered graph with at least one edge cannot have any bridges, therefore, an $F$-cov-sat graph cannot have any acyclic components.
%\end{proof}

In the following theorems, we establish upper bounds on $\cs(F)$ for certain cases of $F$ with a small subgraph that has few neighbors in the rest of the graph. 

\begin{thm} \label{NuNw}
	Suppose $F$ has edge $uw$ with $|N(u) \cup N(w)| = k + 2$. Then
\[ \o{\cs}(F) \leq \left \{ \begin{matrix} 
	\displaystyle k + 1/2, & \delta(F) = k+1; \\ 
	\displaystyle k, & \text{otherwise.} 
\end{matrix}    \right. \]
\end{thm}

\begin{proof}
Let $G_n$ be the $n$-vertex graph attained by fixing $k$ vertices in a copy of $K_{|F|}$, and adding an edge from every one of the $k$ vertices to every one of the other $n - |F|$ vertices. 
That is, $G$ is the result of attaching the complete bipartite graph $K_{n-|F|, k}$ to $k$ vertices in a clique $K_{|F|}$.
(See the left construction in Figure~\ref{fig:NuNw}.)
We see $G_n$ is $F$-semi-saturated (for all sufficiently large $n$) since any added edge would connect two of the $n-|F|$ vertices, and so serve as $uw$ in a new copy of $F$, with the rest of the vertices falling in the $|F|$-clique. 

So long as some vertex in $F$ has degree at most $k$, $G_n$ is also $F$-covered, giving the bound $\cs(F) \leq k$. 
The only issue is when $\delta(F) = k+1$ (then $u$ and $w$ have the same closed neighborhood). 
In this case, simply obtain an $F$-covered graph from $G_n$ by partitioning the $n-|F|$ vertices into pairs and adding an edge to each pair, thus increasing the asymptotic edge density by $1/2$. 
(See the right construction in Figure~\ref{fig:NuNw}.)

\begin{figure}[ht]
\begin{tikzpicture}[scale=.8]
	\draw (0,-.3) ellipse(1.6 and 2);
	\draw (0,-1.4) node[below]{$K_{|F|}$};
	\foreach \x in {170, 130, 50} {
		\filldraw[black!50!white] (\x-35:1.3) -- (\x+15:1.3) -- (\x-10:3) -- (\x-35:1.3);
		\filldraw[black!50!white] (\x-15:1.3) -- (\x+35:1.3) -- (\x+10:3) -- (\x-15:1.3);
		\filldraw (\x+10:3) circle(1mm);
		\filldraw (\x-10:3) circle(1mm);
	}
	\foreach \x in {35,45,55} {
		\filldraw (45+\x:2.5) circle(.5mm);
	}
	\filldraw[fill=white] (0:0) node{$K_{k}$} circle(1.3);
	
	\draw (3.4,0) node{vs.};
	
	\begin{scope}[xshift=8cm]
	\draw (0,-.3) ellipse(1.6 and 2);
	\draw (0,-1.4) node[below]{$K_{|F|}$};
		\foreach \x in {175, 130, 50} {
			\filldraw[black!50!white] (\x-35:1.3) -- (\x+15:1.3) -- (\x-10:3) -- (\x-35:1.3);
			\filldraw[black!50!white] (\x-15:1.3) -- (\x+35:1.3) -- (\x+10:3) -- (\x-15:1.3);
			\draw (\x-10:3)--(\x+10:3);
			\filldraw (\x+10:3) circle(1mm);
			\filldraw (\x-10:3) circle(1mm);
		};
		\foreach \x in {35,45,55} {
			\filldraw (45+\x:2.5) circle(.5mm);
		}
	\filldraw[fill=white] (0:0) node{$K_{k}$} circle(1.3);
	\end{scope}
\end{tikzpicture}
\caption{Graphs realizing the upper bounds in Theorem~\ref{NuNw}: $G_n$ for the general case on the left; for the special case when $\delta(F)=k+1$ on the right.} \label{fig:NuNw}
\end{figure}
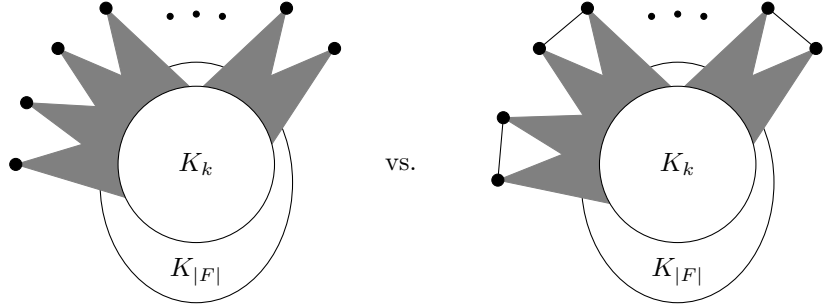

\end{proof}

The weaker bound of Theorem~\ref{NuNw} can be directly generalized to $F$ with two small disjoint vertex sets that only have one edge between them and together have a bounded number of neighbors in the rest of $F$.

\begin{thm}
	Suppose $F$ has disjoint vertex sets $U$ and $W$ such that $|U \cup W| = r$, $|e(U,W)| = 1$, and $(N(U) \cup N(W)) \setminus (U \cup W) = k$. 
	Then $\o{\cs}(F) \leq k + \frac{r-1}{2}$. 
\end{thm}

\begin{proof}
	Taking $G_n$ of the previous proof, partition the $n-|F|$ vertices into $r$-sets and add edges to those $r$-sets to form $r$-cliques. 
	This increases the asymptotic edge density by ${r \choose 2}/r = \frac{r-1}{2}$. 
	Since $|U \cup W| = r$ and $(N(U) \cup N(W)) \setminus (U \cup W) = k$, this new graph is still $F$-covered. 
	And since $|e(U,W)| = 1$, it is also $F$-semi-saturated. 
\end{proof}

%%%%%%%%%%%%%%%%%%%%%%%%%%%%%%%
\section{Graph Classes}

Having established preliminary theory and various structural bounds, let us investigate the cov-sat numbers for some fundamental classes of graphs: paths, cycles, and stars. 
For paths, we first need the following technical lemma. 

%A vertex $c$ of a graph $G$ is a \emph{center} vertex if every vertex in $G$ is distance at most $rad(G)$ from $c$.
%
%\begin{thm}
%	Let $T$ be a tree with at least one leaf $v$ with a degree $2$ neighbor and with distance $1+rad(T-v)$ from a center vertex of $T-v$.
%	Then $\cs(T) < 1$. 
%\end{thm}
%
%\begin{proof}
%	It suffices to construct $T$-cov-sat forests. 
%	Consider a disjoint union of copies of the full $\Delta(T)$-ary tree of height $rad(T-v)$.
%\end{proof}

\begin{lem} \label{T3j}
Let $T$ be a tree on $t \geq 3$ vertices with $t < 3j$ for some $j$. 
Then either $T$ is a star or there exist distinct, non-adjacent vertices $u,v \in V(T)$ such that $T-u-v$ contains no $j$-vertex path with an endpoint in $N(u) \cup N(v)$.
\end{lem}

\begin{proof}
{\bf Case~1:} If $T$ is a star, there is nothing to prove. 
Henceforth, we can assume $T$ has a path on at least $4$ vertices. 

{\bf Case 2:} If there exists a vertex $u$ so that $T-u$ has no $j$-vertex path, let $v$ be any vertex not adjacent to $u$ and we are done. 

{\bf Case 3:} For every edge $e \in E(T)$, if one of the components of $T-e$ has no $j$-vertex path, orient the edge away from that component. 
(We need not worry about neither component having a $j$-vertex path as that was covered in Case 2.) 
Observe that the non-oriented edges form a connected subgraph $U$ of $T$ and that each component of the oriented subgraph has a unique sink that is a leaf of $U$. 
Let $ab$ be an edge in $U$ for some leaf $b$ of $U$. 
Then the component of $T-ab$ that contains $b$ has a $j$-vertex path, else $ab$ would have been oriented. 
Therefore, since $t < 3j$, $U$ has at most $2$ leaves. 
That is, $U$ is itself a path; call the endpoints $u$ and $w$. 
Since we are beyond Case 2, we can assume $u \neq w$. 
Note that $T-u-w$ contains no $j$-vertex path. 
The only remaining issue is if $u$ and $w$ are neighbors. 

{\bf Case 3b:} Assumes $U$ only consists of the edge $uw$. 
Without loss of generality, assume the component $W$ of $T-uw$ that contains $w$ has at most $t/2 < 3j/2$ vertices.
We know that every $j$-vertex path in $W$ contains $w$. 
If no $j$-vertex path in $W$ has $w$ as an endpoint (e.g., the left graph in Figure~\ref{fig:T3j}), then let $v$ be any neighbor of $w$ in $W$ and we are done.
Otherwise, let $v$ be such that some $j$-vertex path in $W$ has endpoint $w$ and final edge $vw$, so the component of $T-u-v$ containing $w$ has at most $(j+1)/2 < j$ vertices (e.g., the right graph in Figure~\ref{fig:T3j}).
\end{proof}

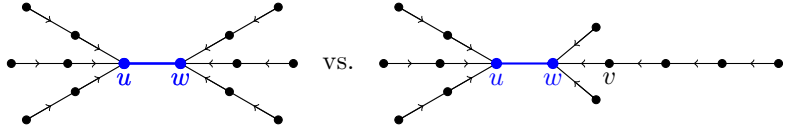
\begin{figure}[ht]
\begin{tikzpicture}[scale=.75]
\foreach \s in {-1,1}{
	\begin{scope}[xshift=\s*.5 cm]
		\foreach \t in {60,90,120}{
			\foreach \r in {0,1}{
				\filldraw  (\t - \s*90:\r) -- (\t - \s*90:\r + 1) circle(.07);
				\draw[very thin, ->] (\t - \s*90:\r + 1) -- (\t - \s*90:\r + .5);
			}
		}
	\end{scope}
	\filldraw[thick,blue] (-.5,0) node[below]{$u$} circle(.08) -- (.5,0) node[below]{$w$} circle(.08);
}

\draw (3.3,0) node{vs.};
	
\begin{scope}[xshift=6.6 cm]
	\begin{scope}[xshift=-.5 cm]
		\foreach \t in {150,180,210}{
			\foreach \r in {0,1}{
				\filldraw  (\t:\r) -- (\t:\r + 1) circle(.07);
				\draw[very thin, ->] (\t:\r + 1) -- (\t:\r + .5);
			}
		}
	\end{scope}
	\begin{scope}[xshift=.5 cm]
		\foreach \t in {-40,40}{
			\filldraw  (\t:0) -- (\t:1) circle(.07);
			\draw[very thin, ->] (\t:1) -- (\t:.5);
		}
	\foreach \x in {1,2,3,4}{
		\filldraw  (\x-1,0) -- (\x,0) circle(.07);
		\draw[very thin, ->] (\x,0) -- (\x-.5,0);
	}
	\draw (1,0) node[below]{$v$};
	\end{scope}
	\filldraw[thick,blue] (-.5,0) node[below]{$u$} circle(.08) -- (.5,0) node[below]{$w$} circle(.08);
\end{scope}

\end{tikzpicture}
\caption{For the proof of Lemma~\ref{T3j}, examples of trees $T$ with $t = 14$ and $j = 5$ in which $U = T[\{u,w\}]$: on the left, observe how every $5$-vertex path in $T-u$ contains $w$ but not as an end-point; on the right, observe how the component of $T-u-v$ that contains $w$ has $3 \leq (j+1)/2$ vertices.} \label{fig:T3j}
\end{figure}

%\begin{lem}
%For the path $P_r$ on $r$ vertices, $\cs(P_3) = \frac{2}{3}$ and $\cs(P_4) = \frac{4}{5}$. 
%\end{lem}
%
%\begin{proof}
%In a $\cs(P_3)$-saturated graph, there's at most one isolated vertex. 
%Note that a graph consisting of disjoint copies of $P_3$ (and perhaps a single isolated vertex) is both $\cs(P_3)$-saturated and $\cs(P_3)$-covered. 
%In a $\cs(P_3)$-covered graph, every component has at least $3$ vertices. 
%Since edge density in graphs on at least $3$ vertices is minimized by $P_3$, the aforementioned construction is best possible. \\
%
%Similarly, disjoint copies of $P_4$ comprise a $P_4$-covered graph, however, not a $P_4$-saturated graph. 
%It is the case that a new edge connecting two disjoint copies of $P_4$ does create additional copies of $P_4$. 
%However, we see that saturation fails when we add an edge connecting two vertices at distince $2$ within a single copy of $P_4$.
%FALSE!!!
%\end{proof}

K\'aszonyi and Tuza~\cite{KT-86} found the saturation number for paths for all sufficiently large $n$. Their result gives, for $r \geq 3$,
\[ \sat(P_r) = \left\{ \begin{matrix} 
%	\frac{2^j - \frac32}{2^j - 1} & \text{if }r = 2j+1;\\
%	\frac{2^j - 1}{2^j - \frac23} & \text{if }r = 2j+2.
	1 - \frac{1}{2\cdot2^j - 2} & \text{if }r = 2j+1;\\
	1 - \frac{1}{3\cdot2^j - 2} & \text{if }r = 2j+2.
 \end{matrix}\right. \]
 We attain a reminiscent result---parity dependent and approaching $1$ monotonically from below as $j$ grows---for the cov-sat density of paths. 

\begin{thm} For path $P_r$ on $r \geq 3$ vertices, 
\[ \cs(P_r) = \left \{ \begin{matrix} 
%	\displaystyle \frac{3j - 1}{3j}, & r = 2j+1; \\ 
%	\displaystyle \frac{3j}{3j + 1}, & r = 2j+2. 
	1 - \frac{1}{3j}, & r = 2j+1; \\ 
	1 - \frac{1}{3j + 1}, & r = 2j+2. 
\end{matrix}    \right. \]
\end{thm}

\begin{proof}
We will prove the first case, $r = 2j+1$. 
The other case follows a nearly identical proof. 

{\bf Claim:} Disjoint copies of $P_{3j}$ comprise an extremal graph for $P_{2j+1}$.

This proposed graph is clearly $P_r$-covered. To check semi-saturation, we need to consider two cases: an edge connecting two copies of $P_r$ and an edge connecting two non-neighbors within a single copy of $P_r$. The first case is trivial. We demonstrate the latter case in Figure \ref{fig:csPr}.

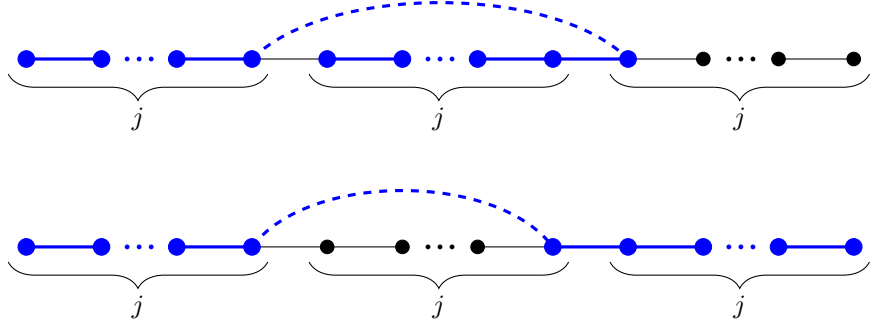
\begin{figure}[ht]
\begin{tikzpicture}
	\foreach \y in {0, 2.5} {
            	\foreach \x in {0,1,2,3,4,5,6,7,8,9,10,11} {
            		\filldraw (\x,\y) circle(.9mm);
            	};
            	\foreach \x in {0,2,3,4,6,7,8,10} {
            		\draw [thin] (\x,\y)--(\x+1,\y);
            	};
            	\foreach \dx in {-1,0,1} {
            		\filldraw[blue] (1.5 + .16*\dx, \y) circle(.25mm);
            		\filldraw (5.5 + \y*1.6 + .16*\dx, \y) circle(.25mm);
            		\filldraw[blue] (9.5 - \y*1.6 + .16*\dx, \y) circle(.25mm);
		};
	\foreach \x in {0,4,8} {
		\draw [decorate,decoration={brace,amplitude=10pt},xshift=-4pt,yshift=0pt] (\x+3.35,\y-.2) -- (\x-.1,\y-.2) node [black,midway,yshift=-6mm] {$j$};
	};
	};
		
	\draw [dashed, blue, very thick] (3,2.5) .. controls (4,3.5) and (7,3.5) .. (8,2.5);
	\foreach \x in {0,2,4,6,7} {
		\draw [very thick, blue] (\x,2.5)--(\x+1,2.5);
	};
	\foreach \x in {0,1,2,3,4,5,6,7,8} {
            		\filldraw [blue] (\x,2.5) circle(1.1mm);
	};
		
	\draw [dashed, blue, very thick] (3,0) .. controls (3.75,1) and (6.25,1) .. (7,0);
	\foreach \x in {0,2,7,8,10} {
		\draw [very thick, blue] (\x,0)--(\x+1,0);
	};
	\foreach \x in {0,1,2,3,7,8,9,10,11} {
            		\filldraw [blue] (\x,0) circle(1.1mm);
	};
\end{tikzpicture}

\caption{Demonstration that $P_{3j}$ is sufficiently long to be $P_{2j+1}$-semi-saturated: A new copy of $P_{2j+1}$ (blue, bolded) is obtained when an edge (dashed) is added, whether the newly adjacent vertices were far apart (top) or close together (bottom).}  \label{fig:csPr}
\end{figure}

It remains to show that the proposed graph does in fact minimize edge density. 
A graph with edge density less than $(3j-1)/(3j)$ would necessarily have acyclic components on $s$ vertices for some $r \leq s < 3j$. 
Take one such component $T$. 
Since $T$ is acyclic, we appeal to Lemma~\ref{T3j}.
%Therefore, either $T$ is a star or $T$ has non-adjacent vertices $u$ and $v$ such that $T-u-v$ contains no $j$-vertex path with an endpoint in $N(u) \cup N(v)$. 
If $T$ is a star, $P_r$-coverage fails for $r\geq 4$, and $r=3$ implies $j=1$ which contradicts $r < 3j$. 

Thus, we have vertices $u$ and $v$ such that $T-u-v$ contains no $j$-vertex path with an endpoint in $N(u) \cup N(v)$. 
Now if we try to extend edge $uv$ into a long path in $T+uv$, we can only possibly make a path on $2 + 2(j-1) < 2j+1$ vertices, contradicting $P_{2j+1}$-semi-saturation. 
\end{proof}

F\"uredi and Kim~\cite{FK-12} showed that 
%\[ \frac{r+3}{r+2} \leq \liminf_{n\to\infty} \frac{\sat(n,C_r)}{n} \leq \limsup_{n\to\infty} \frac{\sat(n,C_r)}{n} \leq \frac{r-3}{r-4} \]
\[ 1 + \frac{1}{r+2} \leq \u{\sat}(C_r) \leq \o{\sat}(C_r) \leq 1 + \frac{1}{r-4} \]
and conjectured the upper bound to be the true limit. 
With a similar construction to theirs we gain similar bounds for the cov-sat density of a cycle. 

\begin{thm} \label{Cr}
%	For cycle $C_r$ on $r\geq 4$ vertices, $1 \leq \cs(C_r) \leq \frac{r-2}{r-3}$.
	For cycle $C_r$ on $r\geq 4$ vertices, $1 \leq \u{\cs}(C_r) \leq \o{\cs}(C_r) \leq 1 + \frac{1}{r-3}$.
\end{thm}

\begin{proof} The lower bound comes from Theorem~\ref{deltaF} since $\delta(C_r)=2$. 
For the upper bound, considering the following construction on $n$ vertices, pictured in Figure~\ref{fig:csCr}. 
Fix two vertices on an $\ell$-vertex clique with $r \leq \ell < 2r-3$ and $ \ell \equiv n \mod{r-3}$. 
With the remaining $n-\ell$ vertices, take $\frac{n-\ell}{r-3}$ disjoint $(r-3)$-vertex paths. 
For each path, add a matching between the end-vertices of the path and the two fixed vertices of the clique. 
This $C_r$-cov-sat graph has, in the limit, edge density $\frac{r-2}{r-3}$. 

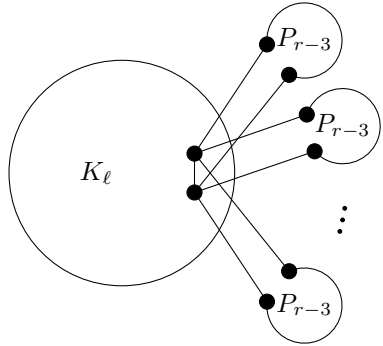
\begin{figure}[ht]
\begin{tikzpicture}
	\draw (0,0) circle(1.5)  node[left]{$K_\ell$} ;
	\filldraw (-15:1) circle(1mm) -- (15:1) circle(1mm);
	\foreach \theta in {36,12, -36} {
		\draw (\theta:3) node{$P_{r-3}$};
		\draw ([shift=(150+\theta:.5)]\theta:3) arc (150+\theta:-150+\theta:.5);
		\foreach \phi in {15,-15} {
			\filldraw ([shift=(10*\phi+\theta:.5)]\theta:3) circle(1mm);
			\draw (\phi:1) -- ([shift=(10*\phi+\theta:.5)]\theta:3);
		};
	};
	\foreach \theta in {-9,-12,-15} {
		\filldraw (\theta:3) circle(.25mm);
	};
\end{tikzpicture}

\caption{A (possibly asymptotically extremal) $C_r$-cov-sat graph.}  \label{fig:csCr}
\end{figure}

\end{proof}

The upper bound in Theorem~\ref{Cr} is realized by a clique with pendant loops on $r-3$ vertices. 
One might try to improve this bound by using longer loops. 
However, the graph is no longer $C_r$-semi-saturated when the loops have, for examples, $r-2$ vertices: Then an edge added between the corresponding vertex in two different loops is contained in no cycle shorter than $r+1$. 

%%This was following a conjecture that the upper bound is best possible, and demonstrating a relation for which cov-sat fails monotonicity:
%If so, observe that the extremal graph for $\cs(C_5)$ is also $G$-cov-sat for the graph $G$ obtained by a $C_4$ and $C_5$ sharing a single edge. But $\cs(G) < \cs(C_4)$, eliminating the possibility of monotonicity for cov-sat numbers, even with connected, induced subgraphs of connected graphs. 

For the star $K_{1,r}$ on $r+1$ vertices, K\'aszonyi and Tuza~\cite{KT-86} also found the saturation number for all $n$, giving
\[ \sat(K_{1,r}) = \frac{r-1}{2}, \]
which matches our lower bound for the cov-sat density of a star. 

\begin{thm} \label{K1r}
	For the star $K_{1,r}$ on $r+1 \geq 3$ vertices, 
%	\[ \frac{r}{2}-\frac{1}{2} \leq \cs(K_{1,r}) \leq \frac{r}{2} - \frac{1}{4} - \frac{1}{8r-4} .\]
	\[ \frac{r-1}{2} \leq \u{\cs}(K_{1,r}) \leq \o{\cs}(K_{1,r}) \leq \frac{r-1}{2} + \frac{4r-3}{8r-4} .\]
\end{thm}

\begin{proof}
	We get the lower bound from $K_{1,r}$-semi-saturation: Any added edge must be incident to a vertex of degree at least $r$, so a $K_{1,r}$-semi-saturated graph has only a few (fewer than $r$) vertices of degree less than $r-1$. 
	%(Indeed, if there were at least $r$ vertices of degree at most $r-2$, then there would be two not adjacent, contradicting semi-saturation.) 
	%But by $K_{1,r}$-coverage, there are 

	Since $K_{1,r}$ has a bridge, an upper bound of $r/2$ follows from Theorem~\ref{bridge}. 
	However, we can do slightly better using disjoint copies of the complete bipartite graph $K_{r-1,r}$, which has edge density $\frac{r\cdot(r-1)}{r + (r-1)} = \frac{r^2 - r}{2r-1}$. % = \frac{r}{2} - \frac{r}{4r-2}$.
\end{proof}

%[Keep the following if the lower bound above is improved:] 
%By applying only $K_{1,r}$-semi-saturation, one obtains a lower bound of $(r-1)/2$, and in fact there exist $K_{1,r}$-semi-saturated graphs with an edge density stritly less than the upper bound above. 
%For example, when $r$ is even, disjoint copies of $K_{r+1}$ with a matching on $r$ vertices removed produce $K_{1,r}$-semi-saturated graphs with edge density $r^2/2(r+1)$.

%%%Uhh... the following construction only gives star-sat graphs, not star-cov
%By specific construction, we can slightly improve the upper bound for stars.
%\begin{thm} \label{csK1r}
%	For star $K_{1,r}$ on $r+1 \geq 3$ vertices, 
%	\[ \cs(K_{1,r}) \leq \left\{ \begin{matrix}
%		\displaystyle \frac{2s^2}{2s+1}, & r = 2s; \\
%		\displaystyle \frac{(2s+1)(s+1)}{2s+3}, & r = 2s+1; \\
%	\end{matrix} \right. \]
%\end{thm}
%\begin{proof}
%The $K_{1,r}$-cov-sat graph consists of disjoint copies of graph $G$ whose construction depends on the parity of $r$. 
%For $r = 2s$, $G$ is $K_{2s+1}$ with a matching removed from $2s$ vertices. 
%For $r = 2s+1$, $G$ is $K_{2s+2}$ with a cycle on $2s$ vertices removed as well as the edge between the other $2$ vertices. 
%\end{proof}

Elegantly, both $\cs(P_r)$ and $\cs(C_r)$ approach $1$ as $r$ approaches infinity. 
That is, long paths and long cycles are similar with respect to the cov-sat invariant. 
On the other hand, stars have cov-sat number near the maximum possible (roughly half the number of vertices) for graphs with a bridge. 
Thus the cov-sat invariant clearly distinguishes the opposite extremes of trees: paths and stars. 
However, we see in the following theorem that this might not be the best way to view cover-saturation. 

\begin{thm} \label{StarPlus}
	Let $G_s$ be the $s$-vertex graph formed by appending an edge onto a leaf of the star $K_{1,s-2}$. 
	Then $1 - \frac{1}{s} \leq \u{\cs}(G_s) \leq \o{\cs}(G_s) \leq 1 - \frac{1}{2s-2}$.
\end{thm}

\begin{proof}
	For the lower bound, observe that (aside from at most one isolated vertex) every connected component of a $G_s$-cov-sat graph has at least $s$ vertices. 
	
	Let $H_s$ be the $(2s-2)$-vertex graph obtained by connecting the centers of two copies of the star $K_{1,s-2}$ (see Figure~\ref{fig:csGs}).
	The upper bound follows from observing that disjoint copies of $H_s$ form a $G_s$-cov-sat graph.
\end{proof}

\begin{figure}[ht]
\begin{tikzpicture}
	\draw (-1.4,0) node{$G_7$};
	\filldraw (0,0) circle(.1) -- (1,0) circle(.1) -- (2,0) circle(.1);
	\foreach \theta in {72, 144, 216, 288}{
		\filldraw (0,0)--(\theta:1) circle(.1);
	}
	
	\begin{scope}[xshift=5cm]
		\filldraw (0,0) circle(.1) -- (1.5,0) circle(.1);
		\foreach \theta in {70, 125, 180, 235, 290}{
			\filldraw (0,0)--(\theta:1) circle(.1);
		}
		\begin{scope}[xshift=1.5cm]
			\foreach \theta in {70, 125, 180, 235, 290}{
				\filldraw (0,0)--(\theta+180:1) circle(.1);
			}
		\end{scope}
		\draw (3.1,0) node{$H_7$};
	\end{scope}

\end{tikzpicture}

\caption{$G_7$ (left) and a $G_7$-cov-sat tree, $H_7$ (right).}  \label{fig:csGs}
\end{figure}
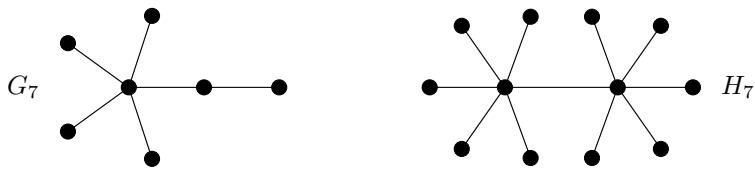

%%%%%%%%%%%%%%%%%%
\section{Conclusion}
Coverage opens a natural saturation variant with numerous potential avenues of further research (see below). 
Moreover, cover-saturation give us a novel graph invariant, a sort of connectivity or centrality measure that is monotone with respect to the coverage relation. 
The cov-sat density of a graph $F$ can only be less than $1$ if $F$ has an acyclic component. 
On the other hand, the cov-sat density of $F$ is at most $|F|-3/2$, a bound realized by cliques. 
%For example, there may be an inverse relation between $\cs(F)$ and the Wiener index for connected $F$. 

%%%%%%%%%%%%%
\subsection{Future Directions}
The sets of extremal graphs for the clique in general saturation and in Tur\'an theory, $\Sat(n,K_r)$ and $\Ex(n,K_r)$ respectively, each contain a unique graph (see~\cite{EHM-64} and~\cite{Tu-41}). 
Is this the case for the analogous set $\Cs(n,K_r)$ of graphs realizing $\cs(n,K_r)$? In particular, is $\Cs(n,K_r) = \{C(n,r)\}$ with the graph $C(n,r)$ as defined in Theorem~\ref{csKr}? 
Even if extremal graphs are not unique, all the present examples are highly symmetric (aside from some small set of vertices). 
Perhaps something can be said of the automorphism group of extremal graphs. 
%[Mike: The lack of uniqueness for extremal graphs in general graph saturation is an interesting quirk.  Do you believe that there is some sort of stability (even in terms of the automorphism group) here?]

It remains to investigate the relationship between the cov-sat number of a graph and other standard saturation numbers ($\sat$, $\ss$, or $\ws$). 
There may also be connections between cov-sat density and other graph measures, such as Wiener index.

Saturation has been extensively studied for families of graphs. 
In fact, Oleg Pikhurko~\cite{Pi-04} identified families $\mathcal{F}$ with as few as $4$ graphs such that $\u{\sat}(\mathcal{F}) \neq \o{\sat}(\mathcal{F})$. 
Additionally, many saturation results have been extended to the $k$-uniform hypergraph setting, where saturations numbers are no longer $O(n)$ but $O(n^{k-1})$. 
What theory arises when cover-saturation is considered for hypergraphs or families of (hyper)graphs? 

In 1991, the first graph saturation game was introduced by F\"uredi, Reimer, and Seress~\cite{FRS-91}, 
 in which two players take turns adding an edge to an initially empty vertex set, one with the aim of constructing an $F$-saturated graph as quickly as possible, and the other as slowly as possible. 
Clearly the game saturation number of $F$ (how long the game lasts when both players playing optimally) lies between $\sat(n,F)$ and $\ex(n,F)$. 
See the 2016 paper by Carraher et al.~\cite{CKRW-16} for a summary of known results on game saturation numbers. 
There may be interesting games to study when the objective of one or both players involves attaining or avoiding $F$-coverage. 

%Poset sat? cov?

%%%%%%%%%%%%%%%%%%
\section*{Acknowledgements}

We would like to thank Mike Ferrara and Claude Tardif for helpful suggestions.

%%%%%%%%%%%%%%%%%%%%%%%%%%%%%%%%%%%%%%%%%%%%%%%%%%%
\bibliography{RsrchRefs}{}
\bibliographystyle{plain}

\end{document}